\documentclass[reqno,a4paper,draft]{amsart}

\usepackage{enumitem}
\setenumerate{label=\textnormal{(\arabic*)}}

\usepackage{amsmath,amssymb,dsfont,verbatim,bm,array,mathtools}
\usepackage[latin1]{inputenc}
\usepackage{booktabs}
\usepackage[raggedright]{titlesec}
\usepackage{mathtools}

\titleformat{\chapter}[display]
{\normalfont\huge\bfseries}{\chaptertitlename\\thechapter}{20pt}{\Huge}
\titleformat{\section}
{\normalfont\Large\bfseries\center}{\thesection}{1em}{}
\titleformat{\subsection}
{\normalfont\large\bfseries}{\thesubsection}{1em}{}
\titleformat{\subsubsection}[runin]
{\normalfont\normalsize\bfseries}{\thesubsubsection}{1em}{}
\titleformat{\paragraph}[runin]
{\normalfont\normalsize\bfseries}{\theparagraph}{1em}{}
\titleformat{\subparagraph}[runin]
{\normalfont\normalsize\bfseries}{\thesubparagraph}{1em}{}
\titlespacing*{\chapter} {0pt}{50pt}{40pt}
\titlespacing*{\section} {0pt}{3.5ex plus 1ex minus .2ex}{2.3ex plus .2ex}
\titlespacing*{\subsection} {0pt}{3.25ex plus 1ex minus .2ex}{1.5ex plus .2ex}
\titlespacing*{\subsubsection}{0pt}{3.25ex plus 1ex minus .2ex}{1.5ex plus .2ex}
\titlespacing*{\paragraph} {0pt}{3.25ex plus 1ex minus .2ex}{1em}
\titlespacing*{\subparagraph} {\parindent}{3.25ex plus 1ex minus .2ex}{1em}

\input xypic
\xyoption{all}

\newtheorem{theorem}{Theorem}[section]
\newtheorem{lemma}[theorem]{Lemma}
\newtheorem{proposition}[theorem]{Proposition}

\theoremstyle{definition}

\newtheorem{notation}[theorem]{Notation}
\newtheorem{example}[theorem]{Example}

\theoremstyle{remark}
\newtheorem{remark}[theorem]{Remark}

\DeclareMathOperator{\Jac}{Jac}



\title{A variation on Magnus' theorem and its generalizations}

\author{Vered Moskowicz}
\address{Department of Mathematics, Bar-Ilan University, Ramat-Gan 52900, Israel.}
\email{vered.moskowicz@gmail.com}

\begin{document}
\begin{abstract}
Let $k$ be a field of characteristic zero, and let $f: k[x,y] \to k[x,y]$, $f: (x,y) \mapsto (p,q)$, 
be a $k$-algebra endomorphism having an invertible Jacobian. 

Write $p=a_ny^n+\cdots+a_1y+a_0$, where
$n=\deg_y(p) \in \mathbb{N}$, $a_i \in k[x]$, $0 \leq i \leq n$, $a_n \neq 0$,
and $q=c_ry^r+\cdots+c_1y+c_0$, where
$r=\deg_y(q) \in \mathbb{N}$, $c_i \in k[x]$, $0 \leq i \leq r$, $c_r \neq 0$.
Denote the set of prime numbers by $P$. 

Under two mild conditions, we prove that, if 
$$
\gcd(\gcd(n,\deg_x(a_n)),\gcd(r,\deg_x(c_r))) \in \{1,8\} \cup P \cup 2P,
$$ 
then $f$ is an automorphism of $k[x,y]$.

Removing (at least one of) the two mild conditions, we present two additional results.
One of the additional results implies that the known form of a counterexample $(P,Q)$ 
to the two-dimensional Jacobian Conjecture,
$l_{1,1}(P)=\epsilon x^{\alpha \mu}y^{\beta \mu}$, $l_{1,1}(Q)=\delta x^{\alpha \nu}y^{\beta \nu}$,
where $\epsilon,\delta \in k^{\times}$,
$1 < \alpha <\beta$, $d:=\gcd(\alpha,\beta) > 1$, $1 < \nu < \mu$, $\gcd(\mu,\nu)=1$,
actually satisfies $d > 2$.
\end{abstract}

\maketitle

\section{Introduction}
Throughout this note, $k$ is a field of characteristic zero
and $f: k[x,y] \to k[x,y]$, $f: (x,y) \mapsto (p,q)$,
is a $k$-algebra endomorphism having an invertible Jacobian,
$\Jac(p,q) \in k^{\times}$.

The famous two-dimensional Jacobian Conjecture, raised by O. H. Keller ~\cite{keller} in 1939,
says that such $f$ is an automorphism of $k[x,y]$.
For more details, see, for example, ~\cite{bcw}, ~\cite{essen affine} and ~\cite{essen book}.

Denote the set of prime numbers by $P$. 

Recall the following results:
\begin{theorem}\label{thm 1}
$f$ is an automorphism of $k[x,y]$ if $\gcd(\deg(p),\deg(q))$:
\begin{itemize}
\item is $1$. 
\item is $\leq 2$. 
\item is $\leq 8$ or belongs to $P$.                   
\item belongs to $2P$.
\end{itemize}
\end{theorem}

In short, Theorem \ref{thm 1} says that $f$ is an automorphism of $k[x,y]$ 
if 
$$
\gcd(\deg(p),\deg(q)) \in \{1,8\} \cup P \cup 2P.
$$

\begin{proof}
\begin{itemize}
\item Magnus ~\cite{magnus}. See also ~\cite[page 158]{nagata}.
\item Nakai-Baba ~\cite{baba nakai}.
\item Appelgate-Onishi ~\cite{app} and Nagata ~\cite[pages 158-159, 169-172]{nagata} ~\cite{nagata2}. 
\item \.Zoladek ~\cite{zoladek} (see also ~\cite{shape}).
\end{itemize}
\end{proof}

Based on the results of Theorem \ref{thm 1}, we have the following:
\begin{theorem}\label{thm 2}
$f$ is an automorphism of $k[x,y]$ if $\deg(p)$ or $\deg(q)$:
\begin{itemize}
\item belongs to $P$. 
\item belongs to $P^2=\{uv\}_{u,v \in P}$. 
\item belongs to $4P$.
\end{itemize}
\end{theorem}

\begin{proof}
\begin{itemize}
\item Indeed, in this case, $\gcd(\deg(p),\deg(q)) \in \{1\} \cup P$, 
and we are done by Theorem \ref{thm 1}. See also ~\cite[Corollary 10.2.25]{essen book}.
\item See ~\cite{app} and ~\cite[pages 169-170, proof of (2)]{nagata}.
\item Indeed, if $\deg(p)=4w$, for some $w \in P$, 
then 
$$
\gcd(\deg(p),\deg(q)) \in \{1,2,4,w,2w,4w\}.
$$ 
Therefore,
\begin{itemize}
\item If $\gcd(\deg(p),\deg(q)) \in \{1,2,4,w,2w\}$, then we are done by Theorem \ref{thm 1}. 
\item If $\gcd(\deg(p),\deg(q))=4w$, then $\deg(p)=4w | \deg(q)$, 
so for some $\lambda \in k$ and $t:=\frac{\deg(q)}{\deg(p)} \in \mathbb{N}$,
we have $\deg(q-\lambda p^t) < \deg(q)$ and we are done by induction on $\deg(q)$
(this argument is the same as that for $\deg(p)$ belongs to $P^2$).
\end{itemize}
\end{itemize}
\end{proof}

Those results are dealing with the total degrees (also called $(1,1)$-degrees) of $p$ and $q$,
$\deg(p)$ and $\deg(q)$, while our results are dealing with the $y$-degrees of $p$ and $q$ 
(also called $(0,1)$-degrees) $\deg_y(p)$ and $\deg_y(q)$,
or with the $x$-degrees of $p$ and $q$ (also called $(1,0)$-degrees), $\deg_x(p)$ and $\deg_x(q)$.

\begin{notation}\label{notation}
We will use the following notations:

\begin{itemize}
\item $p=a_ny^n+\cdots+a_1y+a_0$, where $\deg_y(p)=n \in \mathbb{N}$, $a_i \in k[x]$, $0 \leq i \leq n$, $a_n \neq 0$.
\item $p=b_mx^m+\cdots+b_1x+b_0$, where $\deg_x(p)=m \in \mathbb{N}$, $b_j \in k[y]$, $0 \leq j \leq m$, $b_m \neq 0$.
\item $q=c_ry^r+\cdots+c_1y+c_0$, where $\deg_y(q)=r \in \mathbb{N}$, $c_i \in k[x]$, $0 \leq i \leq r$, $c_r \neq 0$.
\item $q=d_sx^s+\cdots+d_1x+d_0$, where $\deg_x(q)=s \in \mathbb{N}$, $d_j \in k[y]$, $0 \leq j \leq s$, $d_s \neq 0$.
\end{itemize}

\begin{itemize}
\item $A:=\gcd(\deg_y(p),\deg_x(a_n))= \gcd(n,\deg_x(a_n))$.
\item $B:=\gcd(\deg_x(p),\deg_y(b_m))= \gcd(m,\deg_y(b_m))$.
\item $C:=\gcd(\deg_y(q),\deg_x(c_r))= \gcd(r,\deg_x(c_r))$.
\item $D:=\gcd(\deg_x(q),\deg_y(d_s))= \gcd(s,\deg_y(d_s))$.
\end{itemize}

Further denote $u:=\deg_x(a_n)$ and $v:=\deg_x(c_r)$,
so
\begin{itemize}
\item $A= \gcd(n,u)$.
\item $C= \gcd(r,v)$.
\end{itemize}

Write $n=\gcd(n,u)\tilde{n}$, $u=\gcd(n,u)\tilde{u}$,
$r=\gcd(r,v)\tilde{r}$, $v=\gcd(r,v)\tilde{v}$.
Then, $\gcd(\tilde{n},\tilde{u})=1$ and $\gcd(\tilde{r},\tilde{v})=1$.

\end{notation}

Theorem \ref{my thm improved} says that, under two mild conditions,
if 
$$
\gcd(\gcd(n,\deg_x(a_n)),\gcd(r,\deg_x(c_r))) \in \{1,8\} \cup P \cup 2P,
$$ 
then $f$ is an automorphism of $k[x,y]$.
Removing (at least one of) the two mild conditions, we present two additional results.

One of the additional results implies that the known form of a counterexample $(P,Q)$ 
to the two-dimensional Jacobian Conjecture,
$l_{1,1}(P)=\epsilon x^{\alpha \mu}y^{\beta \mu}$, $l_{1,1}(Q)=\delta x^{\alpha \nu}y^{\beta \nu}$,
where $\epsilon,\delta \in k^{\times}$,
$1 < \alpha <\beta$, $d:=\gcd(\alpha,\beta) > 1$, $1 < \nu < \mu$, $\gcd(\mu,\nu)=1$,
actually satisfies $d > 2$.

Our proofs are based on Theorems \ref{thm 1} and \ref{thm 2} and on ideas from number theory
(Dirichlet's theorem on arithmetic progressions and its generalizations).

Of course, we could have replaced $\gcd(A,C)$ by $\gcd(B,D)$ etc., and get analogous results.

\section{Our results}

We begin with an easy observation that will be implicitly used in the proofs of Theorems 
\ref{my thm improved}, \ref{my thm i yes ii no} and \ref{my thm}.

\begin{proposition}\label{prop}
If $n=0$ or $r=0$, then $f$ is an automorphism of $k[x,y]$.
\end{proposition}

Of course, the analogous proposition says that if $m=0$ or $s=0$, then $f$ is an automorphism of $k[x,y]$.

\begin{proof}
W.l.o.g. $n=0$, so $p=a_0 \in k[x]$. 
Then $a_0'q_y=p_xq_y-p_yq_x=\Jac(p,q) \in k^{\times}$,
where $a_0'$ denotes the derivative of $a_0$ with respect to $x$.
This implies that $a_0'=\lambda$ and $q_y=\nu$,
for some $\lambda,\nu \in k^{\times}$.
Therefore, $a_0=\lambda x+\mu$ and $q=\nu y+H(x)$,
for some $\mu \in k$ and $H(x) \in k[x]$.
We obtained that
$f: (x,y) \mapsto (\lambda x+\mu, \nu y+H(x))$,
which is clearly a triangular automorphism of $k[x,y]$.
\end{proof}

Now recall the following nice result due to Dirichlet (1837) ~\cite{dirichlet}, 
which is applied in the proof of Lemma \ref{number theory lemma}.

\begin{theorem}[Dirichlet's theorem on arithmetic progressions]
Let $a,b$ be two positive coprime numbers, namely, $\gcd(a,b)=1$.
Then the arithmetic progression $\{a+nb\}_{n \in \mathbb{N}}$
contains infinitely many prime numbers. 
\end{theorem}

The proof of Theorem \ref{my thm improved} is based on the following lemma:

\begin{lemma}\label{number theory lemma}
Let $a,b,c,d,\epsilon \in \mathbb{N}-\{0\}$ and assume that 
$\gcd(a,b)=1$ and $\gcd(c,d)=1$.
Then there exists $L \in \mathbb{N}$ (actually infinitely many $L$'s) such that: 
\begin{itemize}
\item The maximum of $\{a+Lb,c+Ld\}$ is a prime number.
\item $\gcd(\epsilon,a+Lb)=1$ and $\gcd(\epsilon,c+Ld)=1$.
\end{itemize}
\end{lemma}

\begin{proof}
The proof is due to Erick B. Wong and can be found in ~\cite{mse wong 2} (which uses ~\cite{mse wong 1}).
\end{proof}

Now we are ready to present and prove:

\begin{theorem}\label{my thm improved}
Assume that: 
\begin{itemize}
\item [(i)] $uv \neq 0$.
\item [(ii)] $\tilde{u} \neq \tilde{v}$ or $\tilde{n} \neq \tilde{r}$.
\end{itemize}
If $\gcd(A,C) \in \{1,8\} \cup P \cup 2P$, 
then $f$ is an automorphism of $k[x,y]$.
\end{theorem}

Of course, there exists an analogous result to Theorem \ref{my thm improved}, 
replacing $A$ by $B$, $C$ by $D$, etc.

\begin{proof}
Recall our notations:
\begin{itemize}
\item $p=a_ny^n+\cdots+a_1y+a_0$, where $\deg_y(p)=n$, $a_i \in k[x]$, $0 \leq i \leq n$, $a_n \neq 0$.
\item $q=c_ry^r+\cdots+c_1y+c_0$, where $\deg_y(q)=r$, $c_i \in k[x]$, $0 \leq i \leq r$, $c_r \neq 0$.
\end{itemize}

$u=\deg_x(a_n)$ and $v=\deg_x(c_r)$.

$n=\gcd(n,u)\tilde{n}=A\tilde{n}$, $u=\gcd(n,u)\tilde{u}=A\tilde{u}$,
$r=\gcd(r,v)\tilde{r}=C\tilde{r}$, $v=\gcd(r,v)\tilde{v}=C\tilde{v}$.
($\gcd(\tilde{n},\tilde{u})=1$ and $\gcd(\tilde{r},\tilde{v})=1$).

By our assumption $(i)$, $u \neq 0$ and $v \neq 0$,
hence, we can apply Lemma \ref{number theory lemma} to
$a=\tilde{u}, b=\tilde{n}, c=\tilde{v}, d=\tilde{r}, \epsilon=\gcd(n,u)\gcd(r,v)$,
and get that there exists $L \in \mathbb{N}$ (actually infinitely many $L$'s) such that: 
\begin{itemize}
\item The maximum of $\{\tilde{u}+L\tilde{n},\tilde{v}+L\tilde{r}\}$ is a prime number.
\item $\gcd(\gcd(n,u)\gcd(r,v),\tilde{u}+L\tilde{n})=1$ and 
$\gcd(\gcd(n,u)\gcd(r,v),\tilde{v}+L\tilde{r})=1$.
\end{itemize}

Let 
$$
M_1=\max\{\deg_x(a_n),\ldots,\deg_x(a_1),\deg_x(a_0)\},
$$
$$
M_2=\max\{\deg_x(c_r),\ldots,\deg_x(c_1),\deg_x(c_0)\},
$$
and $M= \max\{M_1,M_2\}$.

By Lemma \ref{number theory lemma}, there exists $L \in \mathbb{N}$ such that:
\begin{itemize}
\item $L > M$.
\item The maximum of $\{\tilde{u}+L\tilde{n},\tilde{v}+L\tilde{r}\}$ is a prime number.
\item $\gcd(\gcd(n,u)\gcd(r,v),\tilde{u}+L\tilde{n})=1$ and 
$\gcd(\gcd(n,u)\gcd(r,v),\tilde{v}+L\tilde{r})=1$.
\end{itemize}

Define $g: (x,y) \mapsto (x,y+x^L)$.

We now consider $gf$ ($gf$ has an invertible Jacobian, by the Chain Rule and by the invertibility of $g$) 
and show that $gf$ is an automorphism of $k[x,y]$,
and then trivially $f$ is an automorphism of $k[x,y]$ (as a product of two automorphisms).

Claim: $\deg((gf)(x))=u+Ln$ and $\deg((gf)(y))=v+Lr$.

Proof of Claim: We will explain why $\deg((gf)(x))=u+Ln$
(same explanation for $\deg((gf)(y))=v+Lr$, with $M_1$ replaced by $M_2$ etc.).

\begin{eqnarray*}
(gf)(x) & = & g(f(x))=g(p)=g(a_ny^n+\cdots+a_jy^j+\cdots+a_1y+a_0) \\ 
& = & g(a_n)g(y)^n+\cdots+g(a_j)g(y)^j+\cdots+g(a_1)g(y)+g(a_0) \\
& = & a_n(y+x^L)^n+\cdots+a_j(y+x^L)^j+\cdots+a_1(y+x^L)+a_0 \\
& = & \sum_{j=0}^{n} p_j,
\end{eqnarray*}
where 
$$
p_j:= a_j(y^j+jy^{j-1}x^L+\cdots+\binom{j}{i}y^{j-i}x^{Li}+\cdots+jyx^{L(j-1)}+x^{Lj}),
$$
$0 \leq j \leq n$.

For a fixed $p_j$, the (total) degrees of 
$$
y^j,jy^{j-1}x^L,\ldots,\binom{j}{i}y^{j-i}x^{Li},\ldots,jyx^{L(j-1)},x^{Lj}
$$
are, respectively, 
$$
j,j-1+L,\ldots,j-i+Li,\ldots,1+L(j-1),Lj.
$$

Trivially, $Lj > j-i+Li$, for every $0 \leq i < j$; 
indeed, $j > i$ implies that $(L-1)j > (L-1)i$,
so $Lj-j > Li-i$, and then $Lj > j+Li-i$.

Therefore, the (total) degree of $p_j$, $\deg(p_j)$, is $\deg_x(a_j)+Lj$.
(The $(1,1)$-leading term is the leading term of $a_j$ multiplied by $x^{Lj}$, 
so it has the form $\lambda_j x^{\deg_x(a_j)+Lj}$, for some $\lambda_j \in k^{\times}$).

Observe that $\deg(p_n) > \deg(p_j)$, for every $0 \leq j < n$;
indeed, by our choice of $L > M \geq M_1$, 
$$
\deg_x(a_j)-\deg_x(a_n) \leq M_1 < L \leq L(n-j),
$$
so 
$$
\deg_x(a_j)-\deg_x(a_n) < L(n-j)=Ln-Lj,
$$
and then 
$$
\deg(p_j)= \deg_x(a_j)+Lj < \deg_x(a_n)+Ln= \deg(p_n).
$$

Concluding that $\deg((gf)(x))=\deg(p_n)=\deg_x(a_n)+Ln=u+Ln$. 
(The $(1,1)$-leading term of $(gf)(x)$ is of the form $\lambda_n x^{u+Ln}$, for some $\lambda_n \in k^{\times}$).

So we have, 
$$
\deg((gf)(x))=u+Ln=A\tilde{u}+LA\tilde{n}=A(\tilde{u}+L\tilde{n})
$$ 
and 
$$
\deg((gf)(y))=v+Lr=C\tilde{v}+LC\tilde{r}=C(\tilde{v}+L\tilde{r}).
$$

Recall that by Lemma \ref{number theory lemma}, we have:
\begin{itemize}
\item The maximum of $\{\tilde{u}+L\tilde{n},\tilde{v}+L\tilde{r}\}$ is a prime number.
\item $\gcd(AC,\tilde{u}+L\tilde{n})=1$ and $\gcd(AC,\tilde{v}+L\tilde{r})=1$.
\end{itemize}

Therefore, it is clear that 
$\gcd(\deg((gf)(x)),\deg((gf)(y)))=\gcd(A,C)$.

$\gcd(A,C) \in \{1,8\} \cup P \cup 2P$,
so $\gcd(\deg((gf)(x)),\deg((gf)(y))) \in \{1,8\} \cup P \cup 2P$.

By Theorem \ref{thm 1}, $gf$ is an automorphism of $k[x,y]$.
\end{proof}

\begin{remark}[Noether's normalization trick]\label{trick}
The trick of defining $g: (x,y) \mapsto (x,y+x^L)$ with large enough $L \in \mathbb{N}$
such that the $(1,1)$-leading term of $(gf)(x)$ is of the form $\nu_1 x^{T_1}$ and
the $(1,1)$-leading term of $(gf)(y)$ is of the form $\nu_2 x^{T_2}$, 
for some $T_1,T_2 > L$ and $\nu_1,\nu_2 \in k^{\times}$,
is sometimes called 'Noether's normalization trick'.
It appears, for example, in ~\cite[Proposition 1.1]{hamann}.
\end{remark}


The two conditions $(i)$ and $(ii)$ in Theorem \ref{my thm improved} 
were necessary in order to conclude that $f$ is an automorphism of $k[x,y]$.

If condition $(i)$ is satisfied but condition $(ii)$ is not satisfied,
namely, $uv \neq 0$, $\tilde{u}=\tilde{v}$ and $\tilde{n}=\tilde{r}$,
then the arguments in the proof of Theorem \ref{my thm improved} show that 
$$
\deg((gf)(x))=u+Ln=A\tilde{u}+LA\tilde{n}=A(\tilde{u}+L\tilde{n})
$$ 
and 
$$
\deg((gf)(y))=v+Lr=C\tilde{v}+LC\tilde{r}=C(\tilde{v}+L\tilde{r})=C(\tilde{u}+L\tilde{n}).
$$
Denote $w:=\tilde{u}+L\tilde{n} \in P$.
Clearly,  
$\gcd(\deg((gf)(x)),\deg((gf)(y)))=\gcd(A,C)w$.

Therefore, if we wish to obtain that $gf$ is an automorphism of $k[x,y]$,
then we should require that $\gcd(A,C) \in \{1,2\}$.    
Indeed, if $\gcd(A,C) \in \{1,2\}$, then 
$\gcd(\deg((gf)(x)),\deg((gf)(y)))=\gcd(A,C)w \in P \cup 2P$,
and we can apply Theorem \ref{thm 1}.

Another option is to apply Theorem \ref{thm 2} to one of $\{\deg((gf)(x)),\deg((gf)(x))\}$,
so we should require that $A \in \{1,4\} \cup P$ or $C \in \{1,4\} \cup P$.
Indeed, if $A \in \{1,4\} \cup P$, then 
$\deg((gf)(x))=Aw \in P \cup 4P \cup P^2$,
and we can apply Theorem \ref{thm 2}.
If $C \in \{1,4\} \cup P$, then 
$\deg((gf)(y))=Cw \in P \cup 4P \cup P^2$,
and we can apply Theorem \ref{thm 2}.

Therefore, we obtained:
\begin{theorem}\label{my thm i yes ii no}
Assume that $uv \neq 0$, $\tilde{u}=\tilde{v}$ and $\tilde{n}=\tilde{r}$.
If one of $\{A,C\}$ belongs to $\{1,4\} \cup P$
or $\gcd(A,C) \in \{1,2\}$,
then $f$ is an automorphism of $k[x,y]$.
\end{theorem}


If condition $(i)$ is not necessarily satisfied, then we have:
\begin{theorem}\label{my thm}
If one of $\{A,C\}$ belongs to $\{1,4\} \cup P$,
then $f$ is an automorphism of $k[x,y]$.
\end{theorem}

Trivially, if one of $\{A,C\}$ belongs to $\{1,4\} \cup P$,
then $\gcd(A,C) \in \{1,4\} \cup P \subset \{1,8\} \cup P \cup 2P$. 

\begin{proof}
\textbf{First case, condition $(i)$ is satisfied:}
There are two options:
\begin{itemize}
\item Condition $(ii)$ is satisfied: Then apply Theorem \ref{my thm improved}.
\item Condition $(ii)$ is not satisfied: Then apply Theorem \ref{my thm i yes ii no}.
\end{itemize}

\textbf{Second case, condition $(i)$ is not satisfied:} 
This means that $u=0$ or $v=0$.
There are two options:
\begin{itemize}
\item Condition $(ii)$ is satisfied: This means that 
$\tilde{u} \neq \tilde{v}$ or $\tilde{n} \neq \tilde{r}$.
If $u=0$ and $v \neq 0$, then 
$A=\gcd(n,u)=\gcd(n,0)=n$, so $\tilde{n}=1$ and $\tilde{u}=0$.

Denote $w:=\tilde{v}+L\tilde{r}$. 
Let $g: (x,y) \mapsto (x,y+x^L)$.
The arguments in the proof of Theorem \ref{my thm improved} show that 
$$
\deg((gf)(x))=u+Ln=Ln=LA
$$ 
and 
$$
\deg((gf)(y))=v+Lr=C\tilde{v}+LC\tilde{r}=C(\tilde{v}+L\tilde{r})=Cw.
$$
(We see that  
$\gcd(\deg((gf)(x)),\deg((gf)(y)))=\gcd(LA,C)$,
which does not help much).

If $A \in \{1,4\} \cup P$, then take large enough $L \in P$,
and get that 
$$
\deg((gf)(x))=LA \in P \cup 4P \cup P^2,
$$
so we can apply Theorem \ref{thm 2}.

If $C \in \{1,4\} \cup P$, then take large enough $L \in \mathbb{N}$,
such that $w=\tilde{v}+L\tilde{r} \in P$ (such $L$ exists by Dirichlet's theorem).
Then,
$$
\deg((gf)(y))=Cw \in P \cup 4P \cup P^2,
$$
so we can apply Theorem \ref{thm 2}.

If $u=v=0$, then take large enough $L \in P$,
and get that
$$
\deg((gf)(x))=u+Ln=Ln=LA
$$ 
and 
$$
\deg((gf)(y))=v+Lr=Lr=LC.
$$
If $A \in \{1,4\} \cup P$, then $\deg((gf)(x))=LA \in P \cup 4P \cup P^2$,
and if $C \in \{1,4\} \cup P$, then $\deg((gf)(y))=LC \in P \cup 4P \cup P^2$,
so we can apply Theorem \ref{thm 2}.

\item Condition $(ii)$ is not satisfied: This means that 
$\tilde{u} = \tilde{v}$ and $\tilde{n} = \tilde{r}$.

If $u=0$, then $A=\gcd(n,u)=\gcd(n,0)=n$, 
so $\tilde{n}=1$ and $\tilde{u}=0$.
Then $\tilde{v}=\tilde{u}=0$, so $v=C\tilde{v}=C0=0$.

Let $g: (x,y) \mapsto (x,y+x^L)$, with large enough $L \in P$.
The arguments in the proof of Theorem \ref{my thm improved} show that 
$$
\deg((gf)(x))=u+Ln=Ln=LA
$$ 
and 
$$
\deg((gf)(y))=v+Lr=Lr=LC
$$
If $A \in \{1,4\} \cup P$, then $\deg((gf)(x))=LA \in P \cup 4P \cup P^2$,
and we can apply Theorem \ref{thm 2}.
If $C \in \{1,4\} \cup P$, then $\deg((gf)(x))=LC \in P \cup 4P \cup P^2$,
and we can apply Theorem \ref{thm 2}.
\end{itemize}
\end{proof}


\section{Examples}
Observe that there are cases where at least one of Theorems \ref{my thm improved}, \ref{my thm i yes ii no} 
and \ref{my thm} is applicable, while none of Theorems \ref{thm 1} and \ref{thm 2} is,
and vice versa, there are cases where at least one of Theorems \ref{thm 1} and \ref{thm 2} 
is applicable, while none of our theorems is. 
Also, of course, there are cases where our theorems and the original theorems are not applicable.

More elaborately:
\begin{example}[One of our theorems is applicable, while none of the original theorems is]

\textbf{First example:}
Denote the set of prime numbers strictly less than $m$ by $P_m$.
Assume that $p$ is of the following form:
$p=x^m+\sum_{j \in \{1,4\} \cup P_{m}}e_j x^j+e_0$, $e_j \in k[y]$ and $\deg_y(e_0) < \deg_y(p)$.
Then $f$ is an automorphism of $k[x,y]$.
Indeed, it is not difficult to see that
$A \in \{1,4\} \cup P_{m} \subset \{1,4\} \cup P$, so we can apply Theorem \ref{my thm}
and get that $f$ is an automorphism of $k[x,y]$.

\textbf{Second example:}
Let 
$$
f: (x,y) \mapsto (x+y+x^2+y^{15}+2xy^{15}+y^{30},y+x^2+2xy^{15}+y^{30}).
$$
$$
\Jac(p,q)=(1+2x+2y^{15})(1+30xy^{14}+30y^{29})-(1+15y^{14}+30xy^{14}+30y^{29})(2x+2y^{15}).
$$
Denote $\epsilon=2x+2y^{15}$ and $\delta=30xy^{14}+30y^{29}$.

Then,
\begin{eqnarray*}
\Jac(p,q) & = & (1+\epsilon)(1+\delta)-(1+15y^{14}+\delta)(\epsilon) \\
& = & 1+\delta+\epsilon+\epsilon\delta-\epsilon-15y^{14}\epsilon-\delta\epsilon \\
& = & 1+\delta-15y^{14}\epsilon=1+30xy^{14}+30y^{29}-15y^{14}(2x+2y^{15}) \\
& = & 1+30xy^{14}+30y^{29}-30xy^{14}-30y^{29}=1.
\end{eqnarray*}

Here, $n=\deg_y(p)=30$, $m=\deg_x(p)=2$, $r=\deg_y(q)=30$, $s=\deg_x(q)=2$,
$a_n=1$, $b_m=1$, $c_r=1$, $d_s=1$.
Therefore, 
\begin{eqnarray*} A= \gcd(n,\deg_x(a_n))=\gcd(30,0)=30, B= \gcd(m,\deg_y(b_m))=\gcd(2,0)=2, 
\end{eqnarray*}
\begin{eqnarray*} C= \gcd(r,\deg_x(c_r))=\gcd(30,0)=30, D= \gcd(s,\deg_y(d_s))=\gcd(2,0)=2.
\end{eqnarray*}

Here, $B=2 \in \{1,4\} \cup P$ (and also $D=2 \in \{1,4\} \cup P$), 
so we can apply (only) Theorem \ref{my thm} and get that $f$ is an automorphism of $k[x,y]$.

It is not possible to apply either one of the original theorems,
since here $\deg(p)=\deg(q)=\gcd(\deg(p),\deg(q))=30$, is a product of three primes.

There are other ways to show that $f$ is an automorphism of $k[x,y]$, 
independent of Theorem \ref{my thm}, for example: 
$p(x,0)=x+x^2$ and $q(x,0)=x^2$, hence $k[p(x,0),q(x,0)]=k[x+x^2,x^2]=k[x]$.
Then ~\cite[Theorem 3.5]{radial similarity} implies that $f$ is an automorphism of $k[x,y]$.
If we further assume that $k$ is algebraically closed, then also ~\cite{inj on one line} 
implies that $f$ is an automorphism of $k[x,y]$.
\end{example}

\begin{example}[One of the original theorems is applicable, while none of our theorems is]
Let 
$$
f: (x,y) \mapsto (x+(x-y)^{15},y+(x-y)^{15}).
$$

$$
\Jac(p,q)=(1+15(x-y)^{14})(1-15(x-y)^{14})-(-15(x-y)^{14})(15(x-y)^{14}).
$$
Denote $w:=15(x-y)^{14}$. 
Then $\Jac(p,q)=(1+w)(1-w)-(-w)(w)=1-w^2+w^2=1$.

Here, $n=\deg_y(p)=15$, $m=\deg_x(p)=15$, $r=\deg_y(q)=15$, $s=\deg_x(q)=15$,
$a_n=1$, $b_m=1$, $c_r=1$, $d_s=1$. 
Therefore,
\begin{eqnarray*}
A= \gcd(n,\deg_x(a_n))=\gcd(15,0)=15, B= \gcd(m,\deg_y(b_m))=\gcd(15,0)=15, \end{eqnarray*}
\begin{eqnarray*} 
C= \gcd(r,\deg_x(c_r))=\gcd(15,0)=15, D= \gcd(s,\deg_y(d_s))=\gcd(15,0)=15. \end{eqnarray*}

Condition $(i)$ is not satisfied, so none of Theorems \ref{my thm improved}
and \ref{my thm i yes ii no} is applicable.
None of $\{A,B,C,D\}$ belongs to $\{1,4\} \cup P$, 
so Theorem \ref{my thm} is not applicable.

However, $\deg(p)=\deg(q)=15$ is a product of two primes,
so one of the original theorems is applicable.

There are other ways to show that $f$ is an automorphism of $k[x,y]$, 
independent of the original theorems, the easiest is just to notice that
$x=p-(p-q)^{15} \in k[p,q]$ and $y=q-(p-q)^{15} \in k[p,q]$.
Another way is the same as in the previous example, namely, 
applying ~\cite[Theorem 3.5]{radial similarity} to 
$k[p(x,0),q(x,0)]=k[x+x^{15},x^{15}]=k[x]$.
\end{example}

\begin{example}[None of our theorems and the original theorems is applicable]
Let 
$$
f: (x,y) \mapsto (x+(x-y)^{30},y+(x-y)^{30}).
$$
\begin{eqnarray*}
\Jac(p,q) & = & (1+30(x-y)^{29})(1-30(x-y)^{29})-(-30(x-y)^{29})(30(x-y)^{29}) \\
& = & (1+w)(1-w)-(-w)(w)=1, \end{eqnarray*}
where $w:=30(x-y)^{29}$.

Here, $n=\deg_y(p)=30$, $m=\deg_x(p)=30$, $r=\deg_y(q)=30$, $s=\deg_x(q)=30$.
$a_n=1$, $b_m=1$, $c_r=1$, $d_s=1$.
Therefore,
\begin{eqnarray*} A= \gcd(n,\deg_x(a_n))=\gcd(30,0)=30, B= \gcd(m,\deg_y(b_m))=\gcd(30,0)=30,
\end{eqnarray*} \begin{eqnarray*} C= \gcd(r,\deg_x(c_r))=\gcd(30,0)=30, D= \gcd(s,\deg_y(d_s))=\gcd(30,0)=30.
\end{eqnarray*}

Condition $(i)$ is not satisfied, so none of Theorems \ref{my thm improved}
and \ref{my thm i yes ii no} is applicable.
None of $\{A,B,C,D\}$ belongs to $\{1,4\} \cup P$, 
so Theorem \ref{my thm} is not applicable.

$\deg(p)=\deg(q)=\gcd(\deg(p),\deg(q))=30$,
is a product of three primes,
so none of the original theorems is applicable.

However, it is easy to see that $f$ is an automorphism of $k[x,y]$:
$$
x=p-(x-y)^{30}=p-(x+(x-y)^{30}-y-(x-y)^{30})^{30}=p-(p-q)^{30} \in k[p,q],
$$
$$
y=q-(x-y)^{30}=q-(x+(x-y)^{30}-y-(x-y)^{30})^{30}=q-(p-q)^{30} \in k[p,q].
$$
\end{example}


\section{An application}

It is known (see, for example, ~\cite[Corollary 7.2]{nagata} ~\cite{lang} ~\cite{makar}) 
that if $f$ is \textit{not} an automorphism of $k[x,y]$, 
then there exists an automorphism $g$ of $k[x,y]$ such that:
\begin{itemize}
\item $\deg(g(p))=\deg(p)$, $\deg(g(q))=\deg(q)$.
\item $l_{1,1}(g(p))=\epsilon x^{\alpha\mu}y^{\beta\mu}$, $l_{1,1}(g(q))=\delta x^{\alpha\nu}y^{\beta\nu}$,
where $\epsilon, \delta \in k^{\times}$,
$1 < \alpha < \beta$, $\gcd(\alpha,\beta) > 1$, $1 < \nu < \mu$, $\gcd(\mu,\nu)=1$.
($l_{1,1}(r)$ denotes the $(1,1)$-leading term of $r \in k[x,y]$).
\item The Newton polygon of $g(p)$ is contained in the rectangle having edges
$\{(0,0),(\alpha\mu,0),(0,\beta\mu),(\alpha\mu,\beta\mu)\}$, and similarly for $g(q)$.
\end{itemize}

Write $\alpha=d\alpha'$ and $\beta=d\beta'$, where $\gcd(\alpha',\beta')=1$ ($d=\gcd(\alpha,\beta) > 1$).

If we 'translate' our previous notations $n,m,r,s,a_n,b_m,c_r,d_s$ to the counterexample $gf$, we get: 
$n=\deg_y(g(p))=\beta\mu$, $m=\deg_x((g(p))=\alpha\mu$, $r=\deg_y(g(q))=\beta\nu$, $s=\deg_x((g(q))=\alpha\nu$, 
$a_n=\epsilon x^{\alpha\mu}$, $b_m=\epsilon y^{\beta\mu}$, $c_r=\delta x^{\alpha\nu}$, $d_s=\delta y^{\beta\nu}$.
$u=\deg_x(a_n)= \alpha\mu$, $v=\deg_x(c_r)= \alpha\nu$.
$A=\gcd(n,u)=\gcd(\beta\mu,\alpha\mu)=\gcd(d\beta'\mu,d\alpha'\mu)=d\mu$,
$C=\gcd(r,v)=\gcd(\beta\nu,\alpha\nu)=\gcd(d\beta'\nu,d\alpha'\nu)=d\nu$,
$B=A$, $D=C$.

Concerning the two conditions $(i)$ and $(ii)$: 
\begin{itemize}
\item [(i)] $uv=\alpha\mu\alpha\nu \neq 0$, so condition $(i)$ is satisfied.
\item [(ii)] $n=d\mu\beta'$ and $u=d\mu\alpha'$, so $\tilde{n}=\beta'$ and $\tilde{u}=\alpha'$.
$r=d\nu\beta'$ and $v=d\nu\alpha'$, so $\tilde{r}=\beta'$ and $\tilde{v}=\alpha'$,
so condition $(ii)$ is \textit{not} satisfied.
\end{itemize}

We can apply Theorem \ref{my thm i yes ii no} and obtain the following:

\begin{proposition}\label{prop d > 2}
If $(g(p),g(q))$ is a counterexample to the two-dimensional Jacobian Conjecture, 
then $\gcd(\alpha,\beta)= d > 2$.
\end{proposition}

\begin{proof}
Otherwise, $d=2$ (it is already known that $d > 1$).
We have computed that $A=d\mu$ and $C=d\nu$.
Therefore, $\gcd(A,C)=\gcd(d\mu,d\nu)=d=2 \in \{1,2\}$.
Then Theorem \ref{my thm i yes ii no} implies that $gf$ is an automorphism of $k[x,y]$,
a contradiction.
\end{proof}

For example, Proposition \ref{prop d > 2} implies that $\gcd(\deg(g(p)),\deg(g(q))) \neq 8$.
Indeed, otherwise,
$$
d(\alpha'+\beta')=\alpha+\beta=\gcd((\alpha+\beta)\mu,(\alpha+\beta)\nu)=\gcd(\deg(g(p)),\deg(g(q)))=8.
$$
There are, apriori, four options for $(d,\alpha'+\beta')$:
\begin{itemize}
\item $(1,8)$: It is known that $d \neq 1$.
\item $(2,4)$: By Proposition \ref{prop d > 2}, $d \neq 2$. 
\item $(4,2)$: $\alpha'+\beta'=2$ is impossible, since $0 \neq d \alpha'=\alpha < \beta= d \beta'$,
so $0 \neq \alpha' < \beta'$.
\item $(8,1)$: $\alpha'+\beta'=1$ is impossible.
\end{itemize}
None of the four options is possible, hence $\gcd(\deg(g(p)),\deg(g(q))) \neq 8$.

However, it is known that $\gcd(\deg(g(p)),\deg(g(q))) \geq 36$
except for two possible cases $\{(75,125),(64,224)\}$; see ~\cite{valqui discarding}.
Unfortunately, Proposition \ref{prop d > 2} is not helpful in discarding these two possible cases
or in raising the bound $36$.

Summarizing, Proposition \ref{prop d > 2} just improves $d$ from $d>1$ to $d>2$,
and we do not know if it can improve other known bounds of a counterexample.

\begin{remark}
Theorem \ref{thm 1} says that $f$ is an automorphism of $k[x,y]$ 
if 
$$
\gcd(\deg(p),\deg(q)) \in \{1,8\} \cup P \cup 2P.
$$
If one will prove that $f$ is an automorphism of $k[x,y]$ 
if $\gcd(\deg(p),\deg(q)) \in 3P$, 
then in Theorem \ref{my thm i yes ii no} we will have 
``$\gcd(A,C) \in \{1,2,3\}$", and then in Proposition \ref{prop d > 2}
we will have ``$d > 3$".

More generally, if one will prove that $f$ is an automorphism of $k[x,y]$ 
if 
$$
\gcd(\deg(p),\deg(q)) \in P^2,
$$ 
then in Theorem \ref{my thm i yes ii no} we will have 
``$\gcd(A,C) \in \{1\} \cup P$", and then in Proposition \ref{prop d > 2}
we will have ``$d  \notin \{1\} \cup P$".

In case that one has proved that $f$ is an automorphism of $k[x,y]$ 
if 
$$
\gcd(\deg(p),\deg(q)) \in P^2, 
$$
then by the same arguments as in the proof of
Theorem \ref{thm 2} (~\cite[pages 169-170, proof of (2)]{nagata}),
we get that $f$ is an automorphism of $k[x,y]$ 
if $\deg(p)$ or $\deg(q)$ belongs to $P^3$.
\end{remark}

Finally, we wonder if it is possible to obtain new results concerning the Jacobian Conjecture
by applying known results from number theory, and perhaps also vice versa.

\section{Acknowledgements}
I would like to thank the three generous MathStackExchange users for their inspiring answers,
in chronological order:
\begin{itemize}
\item Jon Wharf, user Joffan: ~\cite{mse joffan}. 
\item Sungjin Kim, user i707107: ~\cite{mse i707107}.
\item Erick B. Wong, user Erick Wong: ~\cite{mse wong 1} ~\cite{mse wong 2}.
\end{itemize}
In a previous, unpublished version of this paper, Jon Wharf and Sungjin Kim allowed me to use their MSE answers,
while in the current version, Erick B. Wong allowed me to use his MSE answer, 
which proves my Lemma ~\ref{number theory lemma}.

\bibliographystyle{plain}

\end{document}